\documentclass[a4paper,10pt]{amsart}
\usepackage{amsmath,amssymb}
\usepackage{amscd}
\newtheorem{thm}{Theorem}[section]
\newtheorem{prop}[thm]{Proposition}
\newtheorem{lemma}[thm]{Lemma}
\newtheorem{cor}[thm]{Corollary}

\theoremstyle{remark}
\newtheorem{remark}[thm]{Remark}
\newcommand{\id}{{\rm{id}}}
\newcommand{\Ad}{{\rm{Ad}}}
\newcommand{\Hom}{{\rm{Hom}}}

\newcommand{\BC}{\mathbf C}

\newcommand{\BB}{\mathbf B}
\newcommand{\la}{\langle}
\newcommand{\ra}{\rangle}

\title{Coactions of a finite dimensional $C^*$-Hopf algebra on
unital $C^*$-algebras, unital inclusions of unital $C^*$-algebras and the strong Morita equivalence}
\author{Kazunori Kodaka and Tamotsu Teruya}
\address{Department of Mathematical Sciences, Faculty of Science, Ryukyu
\endgraf
University, Nishihara-cho, Okinawa, 903-0213, Japan}
\address{Faculty of Education, Gunma University, 4-2 Aramaki-machi, Maebashi City,
\endgraf
Gunma, 371-8510, Japan}
\address{\sl{E-mail address}: \rm{kodaka@math.u-ryukyu.ac.jp}}
\address{\sl{E-mail address}: \rm{teruya@gunma-u.ac.jp}}

\begin{document}
\maketitle
\begin{abstract}
Let $A$ and $B$ be unital $C^*$-algebras and let $H$ be a finite dimensional $C^*$-Hopf algebra.
Let $H^0$ be its dual $C^*$-Hopf algebra. Let $(\rho, u)$ and $(\sigma, v)$ be twisted
coactions of $H^0$ on $A$ and $B$, respectively. In this paper, we shall show the following theorem:
We suppose that the unital inclusions $A\subset A\rtimes_{\rho, u}H$ and $B\subset B\rtimes_{\sigma, v}H$
are strongly Morita equivalent. If $A'\cap (A\rtimes_{\rho, u}H)=\BC1$, then
there is a $C^*$-Hopf algebra automorphism
$\lambda^0$ of $H^0$ such that the twisted coaction $(\rho, u)$ is
strongly Morita equivalent to the twisted coaction
$((\id_B \otimes\lambda^0 )\circ\sigma \, , \, (\id_B \otimes\lambda^0 \otimes\lambda^0 )(v))$
induced by $(\sigma, v)$ and $\lambda^0$.
\end{abstract}

\section{Introduction}\label{sec:intro}In the previous papers \cite {KT3:equivalence} and
\cite {KT4:morita}, we discussed the strong Morita equivalences for twisted coactions of
a finite dimensional $C^*$-Hopf algebra on unital $C^*$-algebras and unital inclusions of
unital $C^*$-algebras. In this paper, we shall discuss the relation between the strong Morita equivalence
for twisted coactions of a finite dimensional $C^*$-Hopf algebra on unital $C^*$-algebras and the
strong Morita equivalence for the unital inclusions of the unital $C^*$-algebras induced by the twisted coactions
of the finite dimensional $C^*$-Hopf algebra on the unital $C^*$-algebras.
\par
Let us explain the problem in detail. Let $A$ and $B$ be unital $C^*$-algebras. Let $H$ be a finite dimensional
$C^*$-Hopf algebra and $H^0$ its dual $C^*$-Hopf algebra. Let $(\rho, u)$ and $(\sigma, v)$ be
twisted coactions of $H^0$ on $A$ and $B$, respectively. Then we can obtain the unital
inclusions of unital $C^*$-algebras $A\subset A\rtimes_{\rho, u}H$ and $B\subset B\rtimes_{\sigma, v}H$.
In the same way as in \cite [Example]{KT4:morita}, we can see that if $(\rho, u)$ and $(\sigma ,v)$
are strongly Morita equivalent, then the unital nclusions $A\subset A\rtimes_{\rho, u}H$ and
$B\subset B\rtimes_{\sigma, v}H$ are strongly Morita equivalent. In this paper, we shall discuss
the inverse implication. Our main theorem is as follows: We suppose that the unital inclusions
$A\subset A\rtimes_{\rho, u}H$ and $B\subset B\rtimes_{\sigma, v}H$ are strongly Morita
equivalent in the sense of \cite [Definition 2.1]{KT4:morita}. If $A'\cap (A\rtimes_{\rho, u}H)=\BC1$, then
there is a $C^*$-Hopf algebra automorphism
$\lambda^0$ of $H^0$ such that the twisted coaction $(\rho, u)$ is
strongly Morita equivalent to the twisted coaction
$((\id_B \otimes\lambda^0 )\circ\sigma \, , \, (\id_B \otimes\lambda^0 \otimes\lambda^0 )(v))$
induced by $(\sigma, v)$ and $\lambda^0$.
\par
For a unital $C^*$-algebra $A$,
let $M_n (A)$ be the $n\times n$-matrix
algebra over $A$ and $I_n$ denotes the unit element in $M_n (A)$. We identify $M_n (A)$ with
$A\otimes M_n (\BC)$.
\par
Let $A$ and $B$ be $C^*$-algebras and $X$ an $A-B$-bimodule. We denote its left $A$-action and
right $B$-action on $X$ by $a\cdot x$ and $x\cdot b$ for any $a\in A$, $b\in B$, $x\in X$, respectively.
Also, we denote by $\widetilde{X}$ the dual $B-A$-bimodule of $X$ and we denote by $\widetilde{x}$
the element in $\widetilde{X}$ induced by $x\in X$.

\section{Preliminaries}\label{sec:pre}Let $H$ be a finite dimensional $C^*$-Hopf algebra.
We denote its comultipication, counit and antipode by $\Delta$, $\epsilon$ and $S$, respectively.
We shall use Sweedler's notation, $\Delta(h)=h_{(1)}\otimes h_{(2)}$ for any $h\in H$ which suppresses a possible summation when we write comultiplications. We denote by $N$ the dimension of $H$.
Let $H^0$ be the dual $C^*$-Hopf algebra of $H$. We denote its comultiplication, counit and antipode by
$\Delta^0$, $\epsilon^0$ and $S^0$, respectively. There is the distinguished projection $e$ in $H$.
We note that $e$ is the Haar trace on $H^0$. Also, there is the distinguished projection $\tau$
in $H^0$ which is the Haar trace on $H$. Since $H$ is finite dimensional,
$H\cong\oplus_{k=1}^L M_{f_k}(\BC)$ and $H^0 \cong \oplus_{k=1}^K M_{d_k }(\BC)$ as
$C^*$-algebras. Let $\{v_{ij}^k \, | \, k=1,2, \dots L , \, i, j=1,2,\dots, f_k \}$ be a system of matrix
units of $H$. Let $\{w_{ij}^k \, | \, k=1,2,\dots,K ,\, i, j=1,2, \dots , d_k \}$ be a basis of $H$ satisfying
Szyma\'nski and Peligrad's \cite [Theorem 2.2,2]{SP:saturated}, which is called a system of
\sl
comatrix units
\rm
of $H$, that is, the dual basis of the system of matrix units of $H^0$. Also, let
$\{\phi_{ij}^k \, | \, k=1,2,\dots,K, \, i,j=1,2,\dots,d_k \}$ and
$\{\omega_{ij}^k \, | \, k=1,2,\dots, L, \, i, j=1,2,\dots,f_k \}$ be systems of matrix units and
comatrix units of $H^0$, respectively. Let $A$ be a unital $C^*$-algebra. We recall the definition of a twisted coaction $(\rho, u)$ of $H^0$ on $A$ (See \cite {KT1:inclusion}, \cite {KT2:coaction}).
Let $\rho$ be a weak coaction of $H^0$ on $A$ and $u$ a unitary element in $A\otimes H^0 \otimes H^0$.
Then we say that $(\rho, u)$ is a
\sl
twisted coaction
\rm
of $H^0$ on $A$ if
\newline
(1) $(\rho\otimes\id)\circ\rho=\Ad(u)\circ(\id\otimes\Delta^0 )\circ\rho$,
\newline
(2) $(u\otimes 1^0 )(\id\otimes\Delta^0 \otimes\id)(u)=(\rho\otimes\id\otimes\id)(u)
(\id\otimes\id\otimes\Delta^0 )(u)$,
\newline
(3) $(\id\otimes h\otimes\epsilon^0 )(u)=(\id\otimes\epsilon^0 \otimes h)(u)=\epsilon^0 (h)1$ for any $h\in H$.
\par
Let $\Hom (H, A)$ be the linear space of all linear maps from $H$ to $A$.
Then by Sweedler \cite [pp67-70]{Sweedler:Hopf}, it becomes a unital convolution $*$-algebra.
Since $H$ is finite dimensional, $\Hom (H, A)$ is isomorphic to
$A\otimes H^0$. For any element $x\in A\otimes H^0$, we denote by $\widehat{x}$ the
element in $\Hom(H, A)$ induced by $x$. Similarly, we define $\Hom (H\times H, A)$.
We identify $A\otimes H^0 \otimes H^0$ with $\Hom(H\times H, A)$.
For any element $y\in A\otimes H^0 \otimes H^0$, we denote by $\widehat{y}$ the element
in $\Hom(H\times H, A)$ induced by $y$. Furthermore, for a Hilbert $C^*$-bimodule
$X$, let $\Hom (H, X)$ be the linear space of all linear maps from $H$ to $X$.
We also identify $\Hom (H, X)$ with $X\otimes H^0$. For any element $x\in X\otimes H^0$, we denote
by $\widehat{x}$ the element in $\Hom(H, X)$ induced by $x$.
\par
For a twisted coaction $(\rho, u)$, we can consider the twisted action of $H$ on $A$
and its unitary element $\widehat{u}$ defined by
$$
h\cdot_{\rho, u}x=\widehat{\rho(x)}(h)=(\id\otimes h)(\rho(x))
$$
for any $x\in A$, $h\in H$. We call it the twisted action induced by $(\rho, u)$.
Let $A\rtimes_{\rho, u}H$ be the twisted crossed product by the twisted action of $H$
on $A$ induced by $(\rho, u)$. Let $x\rtimes_{\rho, u}h$ be the element in $A\rtimes_{\rho, u}H$
induced by elements $x\in A$, $h\in H$. Let $\widehat{\rho}$ be the dual coaction of $H$
on $A\rtimes_{\rho, u}H$ defined by
$$
\widehat{\rho}(x\rtimes_{\rho, u}h)=(x\rtimes_{\rho, u}h_{(1)})\otimes h_{(2)}
$$
for any $x\in A$, $h\in H$. Let $E_1^{\rho, u}$ be the canonical conditional expectation
from $A\rtimes_{\rho, u}H$ onto $A$ defined by
$$
E_1^{\rho, u}(x\rtimes_{\rho, u}h)=\tau(h)x
$$
for any $x\in A$, $h\in H$. Let $\Lambda$ be the set of all triplets $(i, j, k)$, where
$i, j=1,2,\dots,d_k$ and $k=1,2,\dots,K$ with $\sum_{k=1}^K d_k^2 =N$. Let
$W_I=\sqrt{d_k}\rtimes_{\rho, u}w_{ij}^k$ for any $I=(i, j, k)\in\Lambda$.
By \cite [Proposition 3.18]{KT1:inclusion},
$\{(W_I^{\rho *}, W_I^{\rho})\}_{I\in \Lambda}$ is a quasi-basis for $E_1^{\rho, u}$.
\begin{lemma}\label{lem:pre}With the above notations, for any $\psi\in H^0$,
$$
1\rtimes_{\rho, u}1\rtimes_{\widehat{\rho}}\psi
=\sum_{I\in\Lambda}([\psi\cdot_{\widehat{\rho}}W_I^{\rho*}]\rtimes_{\widehat{\rho}}1^0 )
(1\rtimes_{\rho, u}1\rtimes_{\widehat{\rho}}\tau)(W_I^{\rho}\rtimes_{\widehat{\rho}}1^0 ) .
$$
\end{lemma}
\begin{proof}Since $\{(W_I^{\rho*}, W_I^{\rho})\}_{I\in\Lambda}$ is a quasi-basis
for $E_1^{\rho, u}$ by \cite [Proposition 3.18]{KT1:inclusion},
$$
\sum_{I\in\Lambda}(W_I^{\rho*}\rtimes_{\widehat{\rho}}1^0 )(1\rtimes_{\rho, u}1\rtimes_{\widehat{\rho}}\tau)
(W_I^{\rho}\rtimes_{\widehat{\rho}}1^0 )=1\rtimes_{\rho, u}1\rtimes_{\widehat{\rho}}1^0 .
$$
Hence for any $\psi\in H^0$,
\begin{align*}
1\rtimes_{\rho, u}1\rtimes_{\widehat{\rho}}\psi &
= \sum_{I\in \Lambda}(1\rtimes_{\rho, u}1\rtimes_{\widehat{\rho}}\psi)(W_I^{\rho*}\rtimes_{\widehat{\rho}}1^0 )
(1\rtimes_{\rho, u}1\rtimes_{\widehat{\rho}}\tau)(W_I^{\rho}\rtimes_{\widehat{\rho}}1^0 ) \\
& =\sum_{I\in\Lambda}([\psi_{(1)}\cdot_{\widehat{\rho}}W_I^{\rho*}]\rtimes_{\widehat{\rho}}\psi_{(2)})
(1\rtimes_{\rho, u}1\rtimes_{\widehat{\rho}}\tau)(W_I^{\rho}\rtimes_{\widehat{\rho}}1^0 ) \\
& =\sum_{I\in\Lambda}([\psi\cdot_{\widehat{\rho}}W_I^{\rho*}]\rtimes_{\widehat{\rho}}1^0 )
(1\rtimes_{\rho, u}1\rtimes_{\widehat{\rho}}\tau)(W_I^{\rho}\rtimes_{\widehat{\rho}}1^0 ) .
\end{align*}
Therefore, we obtain the conclusion.
\end{proof}

\section{A left coaction of a finite dimensional $C^*$-Hopf algebra on an equivalence bimodule}
\label{sec:left}Let $A$ and $B$ be unital $C^*$-algebras and let $(\rho, u)$ and $(\sigma, v)$
be twisted coactions of $H^0$ on $A$ and $B$, respectively.
Let $A\rtimes_{\rho, u}H$ and $B\rtimes_{\sigma, v}H$ be the twisted crossed products of $A$
and $B$ by $(\rho, u)$ and $(\sigma, v)$, respectively. We denote them by $C$ and $D$, respectively. Then
we obtain unital inclusions $A\subset C$ and $B\subset D$ of unital $C^*$-algebras. We suppose that
$A\subset C$ and $B\subset D$ are strongly Morita equivalent with respect to
a $C-D$-equivalence bimodule $Y$ and its closed subspace $X$ in the sense of
\cite [Definition 2.1]{KT4:morita}. Also, by \cite [Section 2]{KT4:morita},
there is a conditional expectation $E^X$ from $Y$ onto $X$ with respect to $E_1^{\rho, u}$ and
$E_1^{\sigma, v}$ satisfying Conditions (1)-(6) in \cite [Definition 2.2]{KT4:morita}.
Furthermore, by \cite [Section 6]{KT4:morita}, we can see that the unital inclusions
$C\subset C_1$ and $D\subset D_1$ are strongly Morita equivalent with respect to
the $C_1 -D_1$-equivalence bimodule $Y_1$ and its closed subspace $Y$,
where $C_1 =C\rtimes_{\widehat{\rho}}H^0$ and $D_1 =D\rtimes_{\widehat{\sigma}}H^0$.
As defined in \cite {KT4:morita}, we define $Y_1$ as follows: We regard $C$ and $D$ as
a $C_1 -A$-equivalence bimodule and a $D_1 -B$-equivalence bimodule in the
usual way as in \cite [Section 4]{KT4:morita}, respectively. Let $Y_1 =C\otimes _A X\otimes_B \widetilde{D}$.
Let $E^Y$ be the conditional expectation from $Y_1$ onto $Y$ with respect to
$E_2^{\rho, u}$ and $E_2^{\sigma, v}$ defined by
$$
E^Y (c\otimes x\otimes\widetilde{d})=\frac{1}{N}c\cdot x\cdot d^*
$$
for any $c\in C$, $d\in D$, $x\in X$, where $E_2^{\rho, u}$ and $E_2^{\sigma, v}$
are the canonical conditional expectations from $C_1$ and $D_1$ onto $C$ and $D$
defined by
$$
E_2^{\rho, u}(c\rtimes_{\widehat{\rho}}\psi) =c\psi(e)  , \quad
E_2^{\sigma, v}(d\rtimes_{\widehat{\sigma}}\psi)=d\psi(e) .
$$
for any $c\in C$, $d\in D$, $\psi\in H^0$, respectively. We regard $Y$ as a closed subspace of $Y_1$
by the injective linear map $\phi$ from $Y$ to $Y_1$ defined by
$$
\phi(y)=\sum_{I, J\in \Lambda}W_I^{\rho*}\otimes E^X (W_I^{\rho}\cdot y\cdot W_J^{\sigma*})
\otimes\widetilde{W_J^{\sigma*}}
$$
for any $y\in Y$. In this section, we construct a left coaction of $H$ on $Y$
with respect to $(C, \widehat{\rho})$. First, we define
the bilinear map $`` \triangleright "$ from $H^0 \times Y$ to $Y$ as follows:
For any $\psi\in H^0$, $y\in Y$,
$$
\psi\triangleright y =NE^Y ((1\rtimes_{\rho, u}1\rtimes_{\widehat{\rho}}\psi)\cdot
\phi(y)\cdot (1\rtimes_{\sigma, v}1\rtimes_{\widehat{\sigma}}\tau)) .
$$
\begin{remark}\label{remark:calculus}For any $\psi\in H^0$, $y\in Y$,
$$
\psi\triangleright y=\sum_{I\in \Lambda}[\psi\cdot_{\widehat{\rho}}W_I^{\rho*}]
\cdot E^X (W_I^{\rho}\cdot y) .
$$
Indeed, since $1\rtimes_{\rho, u}1\rtimes_{\widehat{\rho}} \tau$ is the Jones projection
for the conditional expectation $E_1^{\rho, u}$, for any $c\in C$
$(1\rtimes_{\rho, u}1\rtimes_{\widehat{\rho}}\tau)\cdot c=E_1^{\rho, u}(c)$. Also,
$(1\rtimes_{\sigma, v}\rtimes_{\widehat{\sigma}}\tau)\cdot d =E_1^{\sigma, v}(d)$
for any $d\in D$.
Hence by Lemma \ref{lem:pre}, for any $\psi\in H^0$, $y\in Y$,
\begin{align*}
&\psi\triangleright y \\
&=\sum_{I, J\in \Lambda}NE^Y ((1\rtimes_{\rho, u}1\rtimes_{\widehat{\rho}}\psi)\cdot W_I^{\rho*}
\otimes E^X (W_I^{\rho}\cdot y\cdot W_J^{\sigma*} )\otimes\widetilde{W_J^{\sigma*}}\cdot
(1\rtimes_{\sigma, v}1\rtimes_{\widehat{\sigma}}\tau)) \\
& =\sum_{I, I_1 , J \in\Lambda}NE^Y ([\psi\cdot_{\widehat{\rho}}W_{I_1}^{\rho*}]
E_1^{\rho, u}(W_{I_1}^{\rho}W_I^{\rho*})\otimes E^X (W_I^{\rho}\cdot y\cdot W_J^{\sigma*})
\otimes E_1^{\sigma, v}(W_J^{\sigma*})^{\widetilde{}}) \\
& =\sum_{I, I_1 , J\in \Lambda}NE^Y ([\psi\cdot_{\widehat{\rho}}W_{I_1}^{\rho*}]
\otimes E_1^{\rho, u}(W_{I_1}^{\rho}W_I^{\rho*})\cdot E^X (W_I^{\rho}\cdot y\cdot W_{J}^{\sigma*})\cdot
E_1^{\sigma, v}(W_J^{\sigma})\otimes \widetilde{1}) \\
& =\sum_{I, I_1 , J\in\Lambda}[\psi\cdot_{\widehat{\rho}}W_{I_1}^{\rho*}]\cdot E^X (E_1^{\rho, u}(W_{I_1}^{\rho}
W_I^{\rho*})W_I^{\rho}\cdot y\cdot W_J^{\sigma*}E_1^{\sigma, v}(W_J^{\sigma})) \\
& =\sum_{I_1 \in \Lambda}[\psi\cdot_{\widehat{\rho}}W_{I_1}^{\rho*}]\cdot E^X (W_{I_1}^{\rho}\cdot y).
\end{align*}
\end{remark}

\begin{lemma}\label{lem:fix}With the above notations, for any $x\in X$, $\psi\in H^0$,
$\psi\triangleright x=\epsilon^0 (\psi)x$.
\end{lemma}
\begin{proof}By Remark \ref{remark:calculus}, \cite [Theorem 2.2]{SP:saturated} and \cite [Lemma 6.9]{KT4:morita}
\begin{align*}
\psi\triangleright x & =NE^Y ((1\rtimes_{\rho, u}1\rtimes_{\widehat{\rho}}\psi)\cdot\phi(x)\cdot
(1\rtimes_{\sigma, v}1\rtimes_{\widehat{\sigma}}\tau)) \\
& =NE^Y ((1\rtimes_{\rho, u}1\rtimes_{\widehat{\rho}}\psi\tau)\cdot\phi(x)) \\
& =N\epsilon^0 (\psi)E^Y ((1\rtimes_{\rho, u}1\rtimes_{\widehat{\rho}}\tau)\cdot\phi(x))
=\epsilon^0 (\psi)\sum_{I\in\Lambda}[\tau\cdot_{\widehat{\rho}}W_I^{\rho*}]\cdot E^X (W_I^{\rho}\cdot x) \\
& =\epsilon^0 (\psi) ( \sum_{I\in\Lambda}E_1^{\rho, u}(W_I^{\rho*}E_1^{\rho, u}(W_I^{\rho})))\cdot x
=\epsilon^0 (\psi)x .
\end{align*}
\end{proof}

\begin{lemma}\label{lem:definition1}With the above notations, for any $a\in A$, $h\in H$, $\psi\in H^0$, $y\in Y$,
$$
\psi\triangleright(a\rtimes_{\rho, u}h)\cdot y=[\psi_{(1)}\cdot_{\widehat{\rho}}(a\rtimes_{\rho, u}h)]\cdot
[\psi_{(2)}\triangleright y] .
$$
\end{lemma}
\begin{proof}For any $a\in A$, $h\in H$, $\psi\in H^0$ , $y\in Y$,
\begin{align*}
& \psi\triangleright(a\rtimes_{\rho, u}h)\cdot y =
NE^Y ((1\rtimes_{\rho, u}1\rtimes_{\widehat{\rho}}\psi)\cdot
\phi((a\rtimes_{\rho, u}h)\cdot y)\cdot(1\rtimes_{\sigma, v}1\rtimes_{\widehat{\sigma}}\tau)) \\
& =NE^Y ((1\rtimes_{\rho, u}1\rtimes_{\widehat{\rho}}\psi)(a\rtimes_{\rho, u}h\rtimes_{\widehat{\rho}}1^0 )
\cdot \phi(y)\cdot(1\rtimes_{\sigma, v}1\rtimes_{\widehat{\sigma}}\tau)) \\
& =NE^Y (([\psi_{(1)}\cdot_{\widehat{\rho}}(a\rtimes_{\rho, u}h)]\rtimes_{\widehat{\rho}}\psi_{(2)})
\cdot\phi(y)\cdot(1\rtimes_{\sigma, v}1\rtimes_{\widehat{\sigma}}\tau)) \\
& =[\psi_{(1)}\cdot_{\widehat{\rho}}(a\rtimes_{\rho, u}h)]
\cdot NE^Y ((1\rtimes_{\rho, u}
1\rtimes_{\widehat{\rho}}\psi_{(2)})\cdot\phi(y)\cdot (1\rtimes_{\sigma, v}1\rtimes_{\widehat{\sigma}}\tau)) \\
& =[\psi_{(1)}\cdot_{\widehat{\rho}}(a\rtimes_{\rho, u}h)]\cdot[\psi_{(2)}\triangleright y]
\end{align*}
by \cite [Lemma 6.1]{KT4:morita}.
\end{proof}

\begin{lemma}\label{lem:definition2}With the above notations, for any $\psi, \chi\in H^0$, $y\in Y$,
$$
\psi\triangleright[\chi\triangleright y]=\psi\chi\triangleright y .
$$
\end{lemma}
\begin{proof}By Remark \ref {remark:calculus}, for any $\psi, \chi\in H^0$, $y\in Y$,
$$
\psi\triangleright[\chi\triangleright y]=\psi\triangleright\sum_{I\in \Lambda}[\chi\cdot_{\widehat{\rho}}W_I^{\rho*}]
\cdot E^X (W_I^{\rho}\cdot y) .
$$
Since $\chi\cdot_{\widehat{\rho}}W_I^{\rho*}\in C$, by Lemma \ref{lem:definition1},
$$
\psi\triangleright[\chi\triangleright y]=\sum_{I\in\Lambda}[\psi_{(1)}\cdot_{\widehat{\rho}}[\chi\cdot_{\widehat{\rho}}
W_I^{\rho*}]]\cdot[\psi_{(2)}\triangleright E^X (W_I^{\rho}\cdot y)] .
$$
Since $E^X (W_I^{\rho} \cdot y)\in X$, by Lemma \ref {lem:fix} and Remark \ref {remark:calculus}
\begin{align*}
\psi\triangleright[\chi\triangleright y] & =\sum_{I\in\Lambda}[\psi_{(1)}\chi\cdot_{\widehat{\rho}}W_I^{\rho*}]
\cdot\epsilon^0 (\psi_{(2)})E^X (W_I^{\rho}\cdot y) \\
& =\sum_{I\in \Lambda}[\psi\chi\cdot_{\widehat{\rho}}W_I^{\rho*}]\cdot E^X (W_I^{\rho}\cdot y)
=\psi\chi\triangleright y .
\end{align*}
\end{proof}

\begin{lemma}\label{lem:definition3}With the above notations, for any $\psi\in H^0$, $y, z\in Y$,
$$
\psi\cdot_{\widehat{\rho}}{}_C \la y, z \ra ={}_C \la \psi_{(1)}\triangleright y \, , \, S^0 (\psi_{(2)}^* )\triangleright z \ra .
$$
\end{lemma}
\begin{proof}By Remark \ref{remark:calculus} and \cite[Lemma 5.4]{KT4:morita},
for any $\psi\in H^0$, $y, z\in Y$,
\begin{align*}
& {}_C \la \psi_{(1)}\triangleright y \, , \, S^0 (\psi_{(2)}^* )\triangleright z \ra \\
& =\sum_{I, J\in\Lambda}{}_C \la [\psi_{(1)}\cdot_{\widehat{\rho}}W_I^{\rho*}]\cdot E^X (W_I^{\rho}\cdot y) \, , \,
[S^0 (\psi_{(2)}^* )\cdot_{\widehat{\rho}}W_J^{\rho*}]\cdot E^X (W_J^{\rho}\cdot z) \ra \\
& =\sum_{I, J\in\Lambda}[\psi_{(1)}\cdot_{\widehat{\rho}}W_I^{\rho*}]
{}_A \la E^X (W_I^{\rho}\cdot y) \, , \, E^X (W_J^{\rho}\cdot z) \ra [\psi_{(2)}\cdot_{\widehat{\rho}}W_J^{\rho}] \\
& =\sum_{I, J\in\Lambda}[\psi\cdot_{\widehat{\rho}}W_I^{\rho*} \, {}_A
\la E^X (W_I^{\rho}\cdot y) \, , \, E^X (W_J^{\rho}\cdot z ) \ra W_J^{\rho}] \\
& =\sum_{I, J\in\lambda}[\psi\cdot_{\widehat{\rho}} \, {}_C \la W_J^{\rho*}\cdot E^X (W_I^{\rho}\cdot y) \, ,\, 
W_J^{\rho*}\cdot E^X (W_J^{\rho}\cdot z ) \ra ] \\
& =\psi\cdot_{\widehat{\rho}}{}_C \la y, \, z \ra .
\end{align*}
\end{proof}

\begin{prop}\label{prop:leftcoaction}With the above notations, the linear map from $Y$ to $Y\otimes H$ induced
by the bilinear map $ ``\triangleright"$ from $H^0 \times Y$ to $Y$ defined by
$$
\psi\triangleright y =NE^Y ((1\rtimes_{\rho, u}1\rtimes_{\widehat{\rho}}\psi)\cdot \phi(y)\cdot
(1\rtimes_{\sigma, v}1\rtimes_{\widehat{\sigma}}\tau))
$$
for any $y\in Y$, $\psi\in H^0$ is a left coaction of $H$ on $Y$
with respect to $(C, \widehat{\rho})$ satisfying that
$$
\psi\triangleright x = \epsilon^0 (\psi)x
$$
for any $x\in X$, $\psi\in H^0$.
\end{prop}
\begin{proof}By Lemmas \ref{lem:fix}, \ref{lem:definition1}, \ref{lem:definition2} and
\ref{lem:definition3}, it suffices to show that $1^0 \triangleright y=y$
for any $y\in Y$. Indeed, by Remark \ref{remark:calculus} and \cite[Lemma 5.4]{KT4:morita}
$$
1^0 \triangleright y=\sum_{I\in \Lambda}W_I^{\rho*}\cdot E^X (W_I^{\rho}\cdot y)=y
$$
for any $y\in Y$.
\end{proof}

We denote by $\mu$ the above left coaction induced by $``\triangleright"$.

\section{A coaction of a finite dimensional $C^*$-Hopf algebra on a unital $C^*$-algebra induced by a left coaction}\label{sec:coaction}We use the same notations as in the previous section and also we suppose that
the same assumptions as in the previous section. We note that $C=A\rtimes_{\rho, u}H$
and $D=B\rtimes_{\sigma, v}H$. We also note that $D$ is anti-isomorphic to ${}_C \BB(Y)$,
the $C^*$-algebra of all adjointable left $C$-modules maps on $Y$. We identify $D$ with
${}_C \BB(Y)$. Hence for any $T, R\in {}_C \BB(Y)$ we can write their product as follows:
For any $y\in Y$,
$$
(TR)(y)=R(T(y)) .
$$
For any $y, z\in Y$, let $\Theta_{y, z}$ be
the rank-one left $C$-module map on $Y$ defined by
$$
\Theta_{y, z}(x)={}_C \la x, z \ra\cdot y
$$
for any $x\in Y$. Then since $Y$ is of finite type in the sense of Kajiwara
and Watatani \cite{KW1:bimodule}, ${}_C \BB(Y)$ is the linear span of the above rank-one
left $C$-module maps on $Y$ by \cite {KW1:bimodule}. We define a bilinear map
$``\rightharpoonup"$ from $H^0 \times D$ to $D$ so that
$$
\psi\rightharpoonup\la y, z \ra_D =\la S^0 (\psi_{(1)}^* )\triangleright y \, , \, \psi_{(2)}\triangleright z \ra_D
$$
for any $\psi \in H^0$ and $y, z\in Y$. Hence since $\la y, z \ra_D =\Theta_{z, y}$ for any
$y, z\in Y$, we define $``\rightharpoonup"$
as follows: For any $\psi\in H^0$, $y, z\in Y$,
$$
\psi\rightharpoonup \Theta_{y, z}=\Theta_{[\psi_{(2)}\triangleright y]\,, \,[S^0 (\psi_{(1)}^* )\triangleright z]} .
$$

\begin{lemma}\label{lem:identity}With the above notations, for any $\psi\in H^0$ and $T\in {}_C \BB(Y)$,
$$
[\psi\rightharpoonup T](x)=\psi_{(2)}\triangleright T(S^0 (\psi_{(1)})\triangleright x)
$$
for any $x\in Y$.
\end{lemma}
\begin{proof}Since ${}_C \BB(Y)$ is the linear span of the rank-one
left $C$-module maps on $Y$, it suffices to show that
$$
[\psi\rightharpoonup \Theta_{y, z}](x)=\psi_{(2)}\triangleright \Theta_{y, z}(S^0 (\psi_{(1)})\triangleright x)
$$
for any $\psi\in H^0$ and $x, y, z\in Y$. For any $\psi\in H^0$ and $x, y, z\in Y$,
\begin{align*}
[\psi\rightharpoonup\Theta_{y, z}](x) & =\Theta_{[\psi_{(2)}\triangleright y] , [S^0 (\psi_{(1)}^* )\triangleright z]}(x)
={}_C \la x \, , \, S^0 (\psi_{(1)}^* )\triangleright z \ra\cdot [\psi_{(2)}\triangleright y] \\
& ={}_C \la \psi_{(2)}S^0 (\psi_{(1)})\triangleright x \, , \, S^0 (\psi_{(3)}^* )\triangleright z \ra
\cdot [\psi_{(4)}\triangleright y] \\
& =[\psi_{(2)}\cdot_{\widehat{\rho}}{}_C \la S^0 (\psi_{(1)})\triangleright x \, , \, z \ra ]\cdot [\psi_{(3)}\triangleright y] \\
& =\psi_{(2)}\triangleright ({}_C \la S^0 (\psi_{(1)})\triangleright x \, ,\, z \ra\cdot y) \\
& =\psi_{(2)}\triangleright \Theta_{y, z}(S^0 (\psi_{(1)})\triangleright x) .
\end{align*}
Hence we obtain the conclusion.
\end{proof}

\begin{lemma}\label{lem:action}With the above notations, the bilinear map $`` \rightharpoonup "$
from $H^0 \times D$ to $D$ is an action of $H^0$ on $D$.
\end{lemma}
\begin{proof}
Let $T, R\in {}_C \BB(Y)$ and $\psi, \chi \in H^0$.
For any $y\in Y$,
\begin{align*}
([\psi_{(1)}\rightharpoonup T][\psi_{(2)}\rightharpoonup R])(y))
& =[\psi_{(2)}\rightharpoonup R]([\psi_{(1)}\rightharpoonup T](y)) \\
& =\psi_{(3)}\triangleright R(S^0 (\psi_{(2)})\triangleright[\psi_{(1)}\rightharpoonup T](y)) \\
& =\psi_{(4)}\triangleright R(S^0 (\psi_{(3)})\psi_{(2)}\triangleright T(S^0 (\psi_{(1)})\triangleright y)) \\
& =\psi_{(2)}\triangleright R(T(S^0 (\psi_{(1)})\triangleright y)) \\
& =\psi_{(2)}\triangleright(TR)(S^0 (\psi_{(1)})\triangleright y) \\
& =[\psi\rightharpoonup TR](y) .
\end{align*}
Hence we obtain that $\psi\rightharpoonup TR=[\psi_{(1)}\rightharpoonup T][\psi_{(2)}\rightharpoonup R]$.
For any $y\in Y$,
\begin{align*}
[\psi\rightharpoonup[\chi\rightharpoonup T]](y) & =
\psi_{(2)}\triangleright[\chi\rightharpoonup T](S^0 (\psi_{(1)})\triangleright y) \\
& =\psi_{(2)}\triangleright\chi_{(2)}\triangleright T(S^0 (\chi_{(1)})\triangleright S^0 (\psi_{(1)})\triangleright y) \\
& =\psi_{(2)}\chi_{(2)}\triangleright T(S^0 (\psi_{(1)}\chi_{(1)})\triangleright y) \\
& =[\psi\chi\rightharpoonup T](y) .
\end{align*}
Hence we obtain that $\psi\rightharpoonup[\chi\rightharpoonup T]=\psi\chi\rightharpoonup T$.
Let $I_Y$ be the identity map on $Y$. Then for any $y\in Y$
$$
[\psi\rightharpoonup I_Y ](y)=\psi_{(2)}\triangleright I_Y (S^0 (\psi_{(1)})\triangleright y)
=\epsilon^0 (\psi)y .
$$
Hence $\psi\rightharpoonup I_Y =\epsilon^0 (\psi)I_Y$. Also,
for any $y\in Y$, $[1^0 \rightharpoonup T](y)=T(y)$. Thus $1^0 \rightharpoonup T=T$.
Furthermore, for any $y, z\in Y$,
\begin{align*}
{}_C \la [\psi\rightharpoonup T]^* (y) \, , \, z \ra & ={}_C \la y \, , \, [\psi\rightharpoonup T](z) \ra \\
& ={}_C \la y \, , \, \psi_{(2)}\triangleright T(S^0 (\psi_{(1)})\triangleright z) \ra \\
& ={}_C \la S^0 (\psi_{(3)}^* )\psi_{(4)}^* \triangleright y \, , \, \psi_{(2)}
\triangleright T(S^0 (\psi_{(1)})\triangleright z) \ra \\
& = S^0 (\psi_{(2)}^* )\cdot_{\widehat{\rho}}\,{}_C \la \psi_{(3)}^* \triangleright y \, , \, 
T(S^0 (\psi_{(1)})\triangleright z) \ra \\
& = S^0 (\psi_{(2)}^* )\cdot_{\widehat{\rho}}\,{}_C \la T^* (\psi_{(3)}^* \triangleright y) \, , \, 
S^0 (\psi_{(1)})\triangleright z \ra \\
& ={}_C \la S^0 (\psi_{(3)}^* )\triangleright T^* (\psi_{(4)}^* \triangleright y) \, , \, \psi_{(2)}S^0 (\psi_{(1)})
\triangleright z \ra \\
& ={}_C \la S^0 (\psi_{(1)}^* )\triangleright T^* (\psi_{(2)}^* \triangleright y) \, , \, z \ra \\
& ={}_C \la [S^0 (\psi^* )\rightharpoonup T^* ](y) \, ,\, z \ra.
\end{align*}
Hence we can see that $[\psi\rightharpoonup T]^* =S^0 (\psi^* )\rightharpoonup T^* $.
Therefore we obtain the conclusion.
\end{proof}

By Lemma \ref {lem:action}, the map $``\rightharpoonup"$ is an action of $H^0$ on $D$.
We denote by $\beta$ the coaction of $H$ on $D$ induced by the action $``\rightharpoonup"$
of $H^0$ on $D$. By the definition of the action $``\rightharpoonup"$, the left coaction $\mu$
of $H$ on $Y$ induced by the action $``\triangleright"$ is also a right coaction of $H$ on $Y$ with respect
to $(D, \beta)$. Thus $\mu$ is a coaction of $H$ on $Y$ with
respect to $(C, D, \widehat{\rho}, \beta)$. Hence we can see that $\widehat{\rho}$ is strongly
Morita equivalent to $\beta$. By \cite [Section 4]{KT4:morita}, the unital inclusions
$C\subset C_1 (=C\rtimes_{\widehat{\rho}}H^0)$ and $D\subset D\rtimes_{\beta}H^0$ are
strongly Morita equivalent with respect to the $C_1 -D\rtimes_{\beta}H^0$-equivalence
bimodule $Y\rtimes_{\mu}H^0$ and its closed subspace $Y$. Since the unital
inclusions $C\subset C_1$ and $D\subset D_1$ are also strongly Morita equivalent with
respect to the $C_1 -D_1$-equivalence bimodule $Y_1$ and its closed subspace $Y$
by \cite [Corollary 6.3]{KT4:morita}, we can see that the unital inclusions $D\subset D_1$ and
$D\subset D\rtimes_{\beta}H^0 $ are strongly Morita equivalent with respect to
the $D_1 -D\rtimes_{\beta}H^0$-equivalence bimodule
$\widetilde{Y_1}\otimes_{C_1} (Y\rtimes_{\mu}H^0 )$
and its closed subspace $\widetilde{Y}\otimes_C Y\cong D$ by
\cite [Proposition 2.2]{KT4:morita}. 

\section{The exterior equivalence and the strong Morita equivalence}\label{sec:exterior}
First, we shall present some  basic properties on coactions of a finite
dimensional $C^*$-Hopf algebra on a unital $C^*$-algebra and their
exterior equivalence and strong Morita equivalence.
\par
Let $H$ and $K$ be finite dimensional $C^*$-Hopf
algebras and let $H^0$ and $K^0$ be their dual $C^*$-Hopf algebras, respectively.
Let $A$ and $B$ unital $C^*$-algebras and
let $\pi$ be an isomorphism of $B$ onto $A$. Let $\lambda^0$
be a $C^*$-Hopf algebra isomorphism of $K^0$ onto $H^0$.
Let $(\rho, u)$ be a coaction of $K^0$ on $B$ and
let $\rho_{\pi, \lambda^0}$ be the homomorphism of $A$ to $A\otimes H^0$
defined by
$$
\rho_{\pi, \lambda^0}=(\pi\otimes \lambda^0 )\circ\rho\circ\pi^{-1} .
$$
Let $u_{\pi, \lambda^0}$ be the unitary element in $A\otimes H^0 \otimes H^0$ defined by
$$
u_{\pi, \lambda^0}=(\pi\otimes\lambda^0 \otimes\lambda^0 )(u) .
$$

\begin{lemma}\label{lem:induced}With the above notations, $(\rho_{\pi, \lambda^0} ,\, u_{\pi, \lambda^0})$
is a twisted coaction of $H^0$ on $A$.
\end{lemma}
\begin{proof}This is immediate by routine computations.
\end{proof}

We call the above pair $(\rho_{\pi, \lambda^0}, u_{\pi, \lambda^0})$
the twisted coaction of $H^0$ on $A$ induced by
$(\rho, u)$ and $\pi$, $\lambda^0$.
\par
Let $(\sigma, v)$ be a coaction of $K^0$ on $B$ and $(\sigma_{\pi, \lambda^0}, v_{\pi, \lambda^0})$
the twisted coaction of $H^0$ on $A$ induced by $(\sigma, v)$ and $\pi$, $\lambda^0$.

\begin{lemma}\label{lem:exterior}With the above notations, if $(\rho, u)$ is exterior equivalent
$(\sigma, v)$, then $(\rho_{\pi, \lambda^0} ,\, u_{\pi, \lambda^0})$ is exterior equivalent to
$(\sigma_{\pi, \lambda^0}, v_{\pi, \lambda^0})$.
\end{lemma}
\begin{proof}Let $\Delta_H^0$ and $\Delta_K^0$ be the comultiplications of $H^0$ and $K^0$,
respectively. Since $(\rho, u)$ and $(\sigma, v)$ are exterior equivalent, there is a unitary
element $w\in B\otimes K^0$ such that
$$
\sigma =\Ad (w)\circ \rho, \quad
v=(w\otimes 1^0 )(\rho\otimes \id_{K^0})(w)u(\id_B \otimes\Delta_K^0 )(w)^* .
$$
Since $\rho=(\pi\otimes\lambda^0 )^{-1}\circ\rho_{\pi, \lambda^0}\circ\pi$ and
$\sigma=(\pi\otimes\lambda^0 )^{-1}\circ\sigma_{\pi. \lambda^0 }\circ\pi$,
$$
(\pi\otimes\lambda^0 )^{-1}\circ\sigma_{\pi, \lambda^0}\circ\pi
=\Ad(w)\circ(\pi\otimes\lambda^0 )^{-1}\circ\rho_{\pi, \lambda^0}\circ\pi .
$$
Hence
$$
\sigma_{\pi, \lambda^0 }=\Ad((\pi\otimes\lambda^0 )(w))\circ\rho_{\pi, \lambda^0 } .
$$
Furthermore,
\begin{align*}
v_{\pi, \lambda^0}
=(\pi\otimes\lambda^0 \otimes\lambda^0 )(v) 
&=((\pi\otimes\lambda^0 )(w)\otimes 1^0 )(\pi\otimes\lambda^0 \otimes\lambda^0 )((\rho\otimes\id_{K^0})(w)) \\
& \times u_{\pi, \lambda^0 }(\pi\otimes\lambda^0 \otimes\lambda^0 )((\id_B \otimes\Delta_K^0 )(w^* )) .
\end{align*}
Since $(\lambda^0 \otimes\lambda^0 )\circ\Delta_K^0 =\Delta_H^0 \circ\lambda^0 $,
$$
(\pi\otimes\lambda^0 \otimes\lambda^0 )((\id_B \otimes\Delta_K^0 )(w^* ))
=(\id_A \otimes\Delta_H^0 )((\pi\otimes\lambda^0 )(w^* )) .
$$
Since $w\in B\otimes H^0$, we can write that $w=\sum_i b_i\otimes\phi_i $,
where $b_i \in B$, $\phi_i \in H^0$. Then
$$
(\pi\otimes\lambda^0 \otimes\lambda^0 )((\rho\otimes\id_{K^0})(w))
=\sum_i (\pi\otimes\lambda^0 )(\rho(b_i ))\otimes\lambda^0 (\phi_i ) .
$$
On the other hand
\begin{align*}
(\rho_{\pi, \lambda^0}\otimes\id_{H^0})((\pi\otimes\lambda^0 )(w)) & =
(((\pi\otimes\lambda^0 )\circ\rho\circ\pi^{-1})\otimes\id_{H^0})((\pi\otimes\lambda^0 )(w)) \\
& =(((\pi\otimes\lambda^0 )\circ\rho\circ\pi^{-1})\otimes\id_{H^0})(\sum_i \pi(b_i )\otimes\lambda^0 (\phi_i )) \\
& =\sum_i ((\pi\otimes\lambda^0 )\circ\rho)(b_i )\otimes\lambda^0 (\phi_i ) \\
& =\sum_i (\pi\otimes\lambda^0 )(\rho(b_i ))\otimes\lambda^0 (\phi_i ) .
\end{align*}
Thus
$$
(\pi\otimes\lambda^0 \otimes\lambda^0 )((\rho\otimes\id_{K^0})(w))
=(\rho_{\pi, \lambda^0 }\otimes\id_{H^0}((\pi\otimes\lambda^0 )(w)) .
$$
Hence
\begin{align*}
v_{\pi, \lambda^0 } &= ((\pi\otimes\lambda^0 )(w)\otimes 1^0 )(\rho_{\pi, \lambda^0 }\otimes\id_{H^0 })
((\pi\otimes\lambda^0 )(w))u_{\pi, \lambda^0} \\
& \times (\id_A \otimes\Delta_H^0 )((\pi\otimes\lambda^0 )(w^* )) .
\end{align*}Therefore, we obtain the conclusion.
\end{proof}

Let $(\rho, u)$ be a twisted coaction of $H^0$ on $B$ and let $\pi$ be an isomorphism of
$B$ onto $A$. Let
$$
\rho_{\pi}=(\pi\otimes\id)\circ\rho\circ\pi^{-1} , \quad
u_{\pi}=(\pi\otimes\id_{H^0}\otimes\id_{H^0})(u) .
$$
By Lemma \ref{lem:induced}, $(\rho_{\pi}, u_{\pi})$ is a twisted coaction of
$H^0$ on $A$.

\begin{lemma}\label{lem:iso}With the above notations, $(\rho_{\pi}, u_{\pi})$ is
strongly Morita equivalent to $(\rho, u)$.
\end{lemma}
\begin{proof}Let $X_{\pi}$ be the $B-A$-equivalence bimodule induced by
$\pi$, that is, $X_{\pi}=B$ as vector spaces and the left $B$-action on $X_{\pi}$
and the left $B$-valued inner product are defined in the evident way. We define
the right $A$-action and the right $A$-valued inner product as follows: For any
$a\in A$, $x, y\in X_{\pi}$,
$$
x\cdot a=x\pi^{-1}(a) , \quad
\la x, y \ra_A =\pi(x^* y) .
$$
Let $\nu$ be the linear map from $X_{\pi}$ to $X_{\pi}\otimes H^0$
defined by $\nu(x)=\rho(x)$ for any $x\in X_{\pi}$. Then $\nu$ is a twisted coaction of $H^0$
on $X_{\pi}$ with respect to $(B, A, \rho, u, \rho_{\pi}, u_{\pi})$. Indeed,
for any $a\in A$, $b\in B$, $x, y\in X_{\pi}$,
\begin{align*}
& \nu(b\cdot x) =\rho(b)\nu(x)=\rho(b)\cdot \nu(x) , \\
& \nu(x\cdot a) =\nu(x) ((\pi^{-1}\otimes\id)\circ(\pi\otimes\id)\circ\rho\circ\pi^{-1})(a)
=\nu(x)\cdot\rho_{\pi}(a) , \\
& {}_{B\otimes H^0} \la \nu(x), \nu(y) \ra =\rho(xy^* )=\rho({}_B \la x, y \ra) , \\
& \la \nu (x), \nu(y) \ra_{A\otimes H^0 } =(\pi\otimes\id)(\rho(x^* y))=\rho_{\pi}(\la x, y \ra_A ), \\
& (\id\otimes\epsilon^0 )(\nu(x))= (\id\otimes\epsilon^0 )(\rho(x))=x , \\
& ((\nu\otimes\id)\circ\nu)(x)=u(\id\otimes\Delta^0 )(\rho(x))u^*
=u\cdot (\id\otimes\Delta^0 )(\rho(x))\cdot u_{\pi}^* .
\end{align*}
Thus $\nu$ is a twisted coaction of $H^0$ on $X_{\pi}$ with respect to
$(B, A, \rho, u, \rho_{\pi}, u_{\pi})$. Therefore we obtain the conclusion.
\end{proof}

\begin{remark}\label{remark:review}Let $(\rho, u)$ be a twisted coaction of $H^0$ on $A$.
By Lemma \ref{lem:iso}, \cite [Lemma 4.10]{Kodaka:equivariance} and \cite [Theorem 3.3]{KT2:coaction},
we can see that $(\rho, u)$ is strongly Morita equivalent to $\widehat{\widehat{\rho}}$, the second
dual coaction of $(\rho, u)$. This is obtained in \cite [the proof of Corollary 4.8]{KT3:equivalence}.
\end{remark}

Let $A$ and $H$, $H^0$ be as before. Let $(\rho, u)$ and $(\sigma, v)$
be twisted coactions of $H^0$ on $A$. Let $C=A\rtimes_{\rho, u}H$ and $D=A\rtimes_{\sigma, v}H$.
We also regard $A$ as an $A-A$-equivalence bimodule in the usual way. We suppose that the unital inclusions
$A\subset C$ and $A\subset D$ are strongly Morita equivalent with respect to a
$C-D$-equivalence bimodule $Y$ and its closed subspace $A$, that is, we assume that
the $A-A$-equivalence bimodule $A$ is included in $Y$ as a closed subspace.
Furthermore, we suppose that $A'\cap C=\BC1$. Then by \cite [Lemma 10.3]{KT4:morita},
$A' \cap D=\BC1$. Let $E_1^{\rho, u}$ and $E_1^{\sigma, v}$ be the canonical conditional
expectations from $C$ and $D$ onto $A$ defined by
$$
E_1^{\rho, u}(a\rtimes_{\rho, u}h)=\tau(h)a , \quad
E_1^{\sigma, v}(a\rtimes_{\sigma, v}h)=\tau(h)a
$$
for any $a\in A$, $h\in H$. By \cite [Theorem 2.7]{KT4:morita},
there are a conditional expectation $F^A$ of Watatani index-finite type from
$D$ onto $A$ and a conditional expectation $G^A$ from $Y$ onto $A$
with respect to $E_1^{\rho, u}$ and $F^A$. Since $A' \cap D=\BC1$, by
Watatani \cite [Proposition 4.1]{Watatani:index}, $F^A =E_1^{\sigma, v}$. Hence $G^A$ is
a conditional expectation from $Y$ onto $A$ with respect to $E_1^{\rho, u}$ and $E_1^{\sigma, v}$.
By the discussions in \cite [Section 2]{KT4:morita} and the proof of
Rieffel \cite [Proposition 2.1]{Rieffel:rotation}, there is an isomorphism $\Psi$ of $D$ onto $C$
defined by
$$
\Psi(d)={}_C \la 1_A \cdot d \, , \, 1_A \ra
$$
for any $d\in D$, where $A$ is a closed subspace of $Y$ and the unit element in $A$
is regarded as an element in $Y$. Then for any $a\in A$
$$
\Psi(a)={}_A \la 1_A \cdot a \, , \, 1_A \ra={}_A \la a , \, 1_A \ra =a .
$$
Also, for any $d\in D$
\begin{align*}
(E_1^{\rho, u}\circ\Psi)(d) & =E_1^{\rho, u}({}_C \la 1_A \cdot d \, , \, 1_A \ra )
={}_A \la G^A (1_A \cdot d) \, , \, 1_A \ra 
={}_A \la 1_A \cdot E_1^{\sigma, v}(d) \, , \, 1_A \ra \\
& =E_1^{\sigma, v}(d) .
\end{align*}
Thus, we obtain the following lemma:

\begin{lemma}\label{lem:equation}With the above notations, we suppose that
$A' \cap C=\BC1$. Then $\Psi$ is an isomorphism of $D$ onto $C$ satisfying that
$$
\Psi|_A =\id_A, \quad E_1^{\rho, u}\circ\Psi=E_1^{\sigma, v}
$$
\end{lemma}

Let $\widehat{\rho}$ and $\widehat{\sigma}$ be the dual coactions of $H$ on $C$ and $D$
induced by the twisted coactions $(\rho, u)$ and $(\sigma, v)$, respectively and
let $C_1 =C\rtimes_{\widehat{\rho}}H^0$ and $D_1 =D\rtimes_{\widehat{\sigma}}H^0$ the
crossed products of $C$ and $D$ by the actions of $H^0$ on $C$ and $D$ induced by
$\widehat{\rho}$ and $\widehat{\sigma}$, respectively. Similarly, we define $\widehat{\widehat{\rho}}$
and $\widehat{\widehat{\sigma}}$, $C_2 =C_1\rtimes_{\widehat{\widehat{\rho}}}H$ and
$D_2 =D_1 \rtimes_{\widehat{\widehat{\sigma}}}H$, respectively. Let $E_2^{\rho, u}$ and $E_2^{\sigma, v}$
be the canonical conditional expectations from $C_1$ and $D_1$ onto $C$ and $D$ defined in the
same way as above. Also, let $E_3^{\rho, u}$ and $E_3^{\sigma, v}$ be the canonical conditional
expectations from $C_2$ and $D_2$ onto $C_1$ and $D_1$ defined in the same way as above,
respectively. Using $E_1^{\rho, u}$ and $E_1^{\sigma, v}$, we regard $C$ and $D$ as
right Hilbert $A$-modules, respectively. Let $\BB_A (C)$ and $\BB_A (D)$ be the $C^*$-algebras
of all right $A$-module maps on $C$ and $D$, respectively. We note that any right module map
in $\BB_A (C)$ or $\BB_A (D)$ is adjointable by Kajiwara and Watatani
\cite [Lemma 1.10]{KW1:bimodule} since $C$ and $D$ are of finite index in the sense of
\cite {KW1:bimodule}. For any $x, y\in C$, let $\theta_{x, y}^C$ be the rank-one $A$-module
map on $C$ defined by $\theta_{x, y}^C (z)=x\cdot \la y, z \ra_A$ for any $z\in C$.
Similarly, we define $\theta_{x,y}^D $ the rank-one $A$-module map on $D$ for any $x, y\in D$.
Then $\BB_A (C)$ and $\BB_A (D)$ are the linear spans of the above rank-one $A$-module maps
on $C$ and $D$, respectively.
Furthermore, we define $\BB_C (C_1 )$, $\BB_D (D_1 )$ and $\theta_{x, y}^{C_1}$, $\theta_{x',y'}^{D_1}$
for $x, y\in C_1$, $x', y' \in D_1$.

\begin{lemma}\label{lem:equation2}With the above notations and assumptions,
there is an isomorphism $\widehat{\Psi}$ of $D_1$ onto $C_1$ satisfying that
$$
\widehat{\Psi}|_D =\Psi, \quad \Psi\circ E_2^{\sigma, v}=E_2^{\rho, u}\circ\widehat{\Psi}, \quad
\widehat{\Psi}(1\rtimes_{\widehat{\sigma}}\tau)=1\rtimes_{\widehat{\rho}}\tau .
$$
\end{lemma}
\begin{proof}We note that $C_1 \cong\BB_A (C)$ and $D_1 \cong \BB_A (D)$, respectively.
Then $\Psi$ can be regarded as a Hilbert $A-A$-bimodule isomorphism of $D$
onto $C$ since $E_1^{\rho, u}\circ\Psi=E_1^{\sigma, v}$. Let $\widehat{\Psi}$ be the map
from $\BB_A (D)$ to $\BB_A (C)$ defined by $\widehat{\Psi}(T)=\Psi\circ T\circ\Psi^{-1}$ for any
$T\in \BB_A (D)$. By routine computations, $\widehat{\Psi}$ is an isomorphism of $D_1$
onto $C_1$. For any $x\in D$, let $T_x$ be the element in $\BB_A (D)$ defined by $T_x (y)=xy$
for any $y\in D$. Then for any $x\in D$ and $z\in C$,
$$
\widehat{\Psi}(T_x )(z) =(\Psi\circ T_x \circ\Psi^{-1})(z) =
\Psi(x\Psi^{-1}(z))=\Psi(x)z 
=T_{\Psi(x)}(z) .
$$
Hence $\widehat{\Psi}|_D =\Psi$. Also, since $1\rtimes_{\widehat{\rho}}\tau$
and $1\rtimes_{\widehat{\sigma}}\tau$ are identified with $\theta_{1,1}^C$ and $\theta_{1,1}^D$,
respectively. For any $z\in C$,
\begin{align*}
\widehat{\Psi}(1\rtimes_{\widehat{\sigma}}\tau)(z) & =(\Psi\circ\theta_{1,1}^D \circ\Psi^{-1})(z)
=\Psi (1\cdot \la 1, \, \Psi^{-1}(z) \ra_A )=\Psi(E_1^{\sigma, v}(\Psi^{-1}(z))) \\
& =E_1^{\sigma, v}(\Psi^{-1}(z))=E_1^{\rho, u}(z)=\theta_{1,1}^C (z)=(1\rtimes_{\widehat{\rho}}\tau)(z) .
\end{align*}
Thus, $\widehat{\Psi}(1\rtimes_{\widehat{\sigma}}\tau)=1\rtimes_{\widehat{\rho}}\tau$.
Furthermore, for any $x, y\in D$,
\begin{align*}
(E_2^{\rho, u}\circ\widehat{\Psi})((x\rtimes_{\widehat{\sigma}}1^0 )
(1\rtimes_{\widehat{\sigma}}\tau)(y\rtimes_{\widehat{\sigma}}1^0 ))
& =E_2^{\rho, u}((\Psi(x)\rtimes_{\widehat{\rho}}1^0 )
(1\rtimes_{\widehat{\rho}}\tau)(\Psi(y)\rtimes_{\widehat{\rho}}1^0 )) \\
& =\Psi(x)\tau(e)\Psi(y)=\Psi(xy)\tau(e) .
\end{align*}
On the other hand
$$
(\Psi\circ E_2^{\sigma, v})((x\rtimes_{\widehat{\sigma}}1^0 )(1\rtimes_{\widehat{\sigma}}\tau)
(y\rtimes_{\widehat{\sigma}}1^0 ))=\Psi(x\tau(e)y)
=\Psi(xy)\tau(e) .
$$
Hence $E_2^{\rho, u}\circ\widehat{\Psi}=\Psi\circ E_2^{\sigma, v}$.
Therefore, we obtain the conclusion.
\end{proof}

\begin{lemma}\label{lem:equation3}With the above notations and assumptions,
there is an isomorphism $\widehat{\widehat{\Psi}}$ of $D_2$ onto $C_2$ satisfying that
$$
\widehat{\widehat{\Psi}}|_{D_1} =\widehat{\Psi}, \quad
\widehat{\widehat{\Psi}}(1_{D_1}\rtimes_{\widehat{\widehat{\sigma}}}e)
=1_{C_1}\rtimes_{\widehat{\widehat{\rho}}}e , \quad
E_3^{\rho, u}\circ\widehat{\widehat{\Psi}}=\widehat{\Psi}\circ E_3^{\sigma, v} .
$$
\end{lemma}
\begin{proof}We can prove this lemma in the same way as in the proof of Lemma \ref{lem:equation2}
\end{proof}

By Lemmas \ref{lem:equation2} and \ref{lem:equation3}, $\widehat{\Psi}|_{A' \cap D_1}$
and $\widehat{\widehat{\Psi}}|_{D' \cap D_2}$ are isomorphisms of $A' \cap D_1$ and $D' \cap D_2$ onto
$A' \cap C_1$ and $C' \cap C_2$, respectively.

\begin{lemma}\label{lem:Ciso}With the above notations and assumptions,
$A' \cap D_1 \cong H^0$ as $C^*$-algebras.
\end{lemma}
\begin{proof}
Since $H^0 \cong 1\rtimes_{\widehat{\sigma}}H^0$ as $C^*$-algebras, it suffices to show that
$A' \cap D_1 =1\rtimes_{\widehat{\sigma}}H^0$, where we identify $A$ with
$A\rtimes_{\sigma, v}1\rtimes_{\widehat{\sigma}}1^0$. Let $x\in A' \cap D_1$. Then
we can write that $x=\sum_i x_i \rtimes_{\widehat{\sigma}}\phi_i$, where $x_i \in D$
and $\{\phi_i \}_i$ is a basis of $H^0$. For any $a\in A$,
\begin{align*}
(a\rtimes_{\sigma, v}1\rtimes_{\widehat{\sigma}}1^0 )x & =\sum_i (a\rtimes_{\sigma, v}1)x_i
\rtimes_{\widehat{\sigma}}\phi_i , \\
x(a\rtimes_{\sigma, v}1\rtimes_{\widehat{\sigma}}1^0 ) & =\sum_i (x_i \rtimes_{\widehat{\sigma}}\phi_i )
(a\rtimes_{\sigma, v}1\rtimes_{\widehat{\sigma}}1^0 ) \\
& =\sum_i x_i [\phi_{i (1)}\cdot_{\widehat{\sigma}}a\rtimes_{\sigma, v}1]\rtimes_{\widehat{\sigma}}\phi_{i(2)} \\
& =\sum_i x_i (a\rtimes_{\sigma, v}1)\rtimes_{\widehat{\sigma}}\phi_i .
\end{align*}
Since $(a\rtimes_{\sigma, v}1\rtimes_{\widehat{\sigma}}1^0 )x
=x(a\rtimes_{\sigma, v}1\rtimes_{\widehat{\sigma}}1^0 )$,
$(a\rtimes_{\sigma, v}1)x_i =x_i (a\rtimes_{\sigma, v}1)$ for any
$a\in A$ and $i$. Since $A' \cap D=\BC1$, $x_i \in \BC1$ for any $i$.
Therefore, $A' \cap D_1 \subset 1\rtimes_{\widehat{\sigma}}H^0$. The inverse inclusion is
clear. Hence we obtain the conclusion.
\end{proof}

By Lemma \ref{lem:Ciso}, $\widehat{\Psi}|_{A' \cap D_1}$ can be regarded as
a $C^*$-algebra automorphism of $H^0$. We denote it by $\lambda^0$.

\begin{lemma}\label{lem:Hopfiso}With the above notations, $\lambda^0$ is a
$C^*$-Hopf algebra automorphism of $H^0$.
\end{lemma}
\begin{proof}It suffices to show that for any $\phi\in H^0$,
$$
S^0 (\lambda^0 (\phi))=\lambda^0 (S^0 (\phi)), \quad
\Delta^0 (\lambda^0 (\phi))=(\lambda^0 \otimes\lambda^0 )(\Delta^0 (\phi)), \quad
\epsilon^0 (\lambda^0 (\phi))=\epsilon^0 (\phi) .
$$
We note that for any $h\in H$, $\phi\in H^0$,
$$
\phi(h)=N^2 (E_2^{\rho, u}\circ E_3^{\rho, u})((1\rtimes_{\widehat{\rho}}\phi\rtimes_{\widehat{\widehat{\rho}}}1)
(1\rtimes_{\widehat{\rho}}1^0 \rtimes_{\widehat{\widehat{\rho}}}e)
(1\rtimes_{\widehat{\rho}}\tau\rtimes_{\widehat{\widehat{\rho}}}1)
(1\rtimes_{\widehat{\rho}}1^0\rtimes_{\widehat{\widehat{\rho}}}h))
$$
by easy computations. Indeed,
\begin{align*}
& N^2 (E_2^{\rho, u}\circ E_3^{\rho, u})((1\rtimes_{\widehat{\rho}}\phi\rtimes_{\widehat{\widehat{\rho}}}1)
(1\rtimes_{\widehat{\rho}}1^0 \rtimes_{\widehat{\widehat{\rho}}}e)
(1\rtimes_{\widehat{\rho}}\tau\rtimes_{\widehat{\widehat{\rho}}}1)
(1\rtimes_{\widehat{\rho}}1^0\rtimes_{\widehat{\widehat{\rho}}}h)) \\
& =N^2 (E_2^{\rho, u}\circ E_3^{\rho, u})((1\rtimes_{\widehat{\rho}}\phi\rtimes_{\widehat{\widehat{\rho}}}e)
(1\rtimes_{\widehat{\rho}}\tau\rtimes_{\widehat{\widehat{\rho}}}h)) \\
& =N^2 (E_2^{\rho, u}\circ E_3^{\rho, u})((1\rtimes_{\widehat{\rho}}\phi)[e_{(1)}\cdot_{\widehat{\widehat{\rho}}}
(1\rtimes_{\widehat{\rho}}\tau)]\rtimes_{\widehat{\widehat{\rho}}}e_{(2)}h) \\
& =N^2 E_2^{\rho, u}((1\rtimes_{\widehat{\rho}}\phi)[e_{(1)}\cdot_{\widehat{\widehat{\rho}}}
(1\rtimes_{\widehat{\rho}}\tau)])
\tau' (e_{(2)}h) \\
& =N^2 E_2^{\rho, u}((1\rtimes_{\widehat{\rho}}\phi)(1\rtimes_{\widehat{\rho}}\tau_{(1)})\tau_{(2)}(e_{(1)}))
\tau' (e_{(2)}h) \\
& =N^2 E_2^{\rho, u}(1\rtimes_{\widehat{\rho}}\phi\tau_{(1)})\tau_{(2)}(e_{(1)})\tau' (e_{(2)}h) \\
& =N^2 (\phi\tau_{(1)})(e' )\tau_{(2)}(e_{(1)})\tau' (e_{(2)}h) \\
& =N^2 (\phi\tau_{(1)})(e' )\tau_{(2)}(e_{(1)}h_{(2)}S(h_{(1)}))\tau' (e_{(2)}h_{(3)}) \\
& =N^2 (\phi\tau_{(1)})(e' )\tau_{(2)}(e_{(1)}h_{(2)})\tau_{(3)}(S(h_{(1)}))\tau' (e_{(2)}h_{(3)}) \\
& =N^2 (\phi\tau_{(1)})(e' )(\tau_{(2)}\tau' )(eh_{(2)})\tau_{(3)}(S(h_{(1)})) \\
& =N(\phi\tau_{(1)})(e' )\tau_{(2)}(S(h)) \\
& =N(\phi_{(1)}\tau_{(1)})(e' )(S^0 (\phi_{(3)})\phi_{(2)}\tau_{(2)})(S(h)) \\
& =N(\phi_{(1)}\tau_{(1)})(e' )S^0 (\phi_{(3)})(S(h_{(2)}))(\phi_{(2)}\tau_{(2)})(S(h_{(1)})) \\
& =N(\phi_{(1)}\tau)(e' S(h_{(1)}))S^0 (\phi_{(2)})(S(h_{(2)})) \\
& =S^0 (\phi)(S(h)) \\
& =\phi(h) ,
\end{align*}
where $e' =e$ and $\tau' =\tau$. We also note that
$$
\phi(h)=N^2 (E_2^{\sigma, v}\circ E_3^{\sigma, v})
((1\rtimes_{\widehat{\sigma}}\phi\rtimes_{\widehat{\widehat{\sigma}}}1)
(1\rtimes_{\widehat{\sigma}}1^0 \rtimes_{\widehat{\widehat{\sigma}}}e)
(1\rtimes_{\widehat{\sigma}}\tau\rtimes_{\widehat{\widehat{\sigma}}}1^0 )
(1\rtimes_{\widehat{\sigma}}1^0 \rtimes_{\widehat{\widehat{\sigma}}}h))
$$
for any $h\in H$, $\phi\in H^0$. Hence for any $h\in H$, $\phi\in H^0$,
\begin{align*}
& \lambda^0 (\phi)(h) \\
&  =N^2 (E_2^{\rho, u}\circ E_3^{\rho, u})(\widehat{\widehat{\Psi}}
(1\rtimes_{\widehat{\sigma}}\phi\rtimes_{\widehat{\widehat{\sigma}}}1)
(1\rtimes_{\widehat{\rho}}1^0 \rtimes_{\widehat{\widehat{\rho}}}e)
(1\rtimes_{\widehat{\rho}}\tau\rtimes_{\widehat{\widehat{\rho}}}1)
(1\rtimes_{\widehat{\rho}}1^0 \rtimes_{\widehat{\widehat{\rho}}}h)) \\
& =N^2 (E_2^{\rho, u}\circ E_3^{\rho, u}\circ\widehat{\widehat{\Psi}})
((1\rtimes_{\widehat{\sigma}}\phi\rtimes_{\widehat{\widehat{\sigma}}}1)
(1\rtimes_{\widehat{\sigma}}1^0 \rtimes_{\widehat{\widehat{\sigma}}}e)
(1\rtimes_{\widehat{\sigma}}\tau\rtimes_{\widehat{\widehat{\sigma}}}1) \\
& \times\widehat{\widehat{\Psi}}^{-1}(1\rtimes_{\widehat{\rho}}1^0 \rtimes_{\widehat{\widehat{\rho}}}h)) \\
& =N^2 (\Psi\circ E_2^{\sigma, v}\circ E_3^{\sigma, v})
((1\rtimes_{\widehat{\sigma}}\phi\rtimes_{\widehat{\widehat{\sigma}}}1)
(1\rtimes_{\widehat{\sigma}}1^0 \rtimes_{\widehat{\widehat{\sigma}}}e)
(1\rtimes_{\widehat{\sigma}}\tau\rtimes_{\widehat{\widehat{\sigma}}}1) \\
& \times\widehat{\widehat{\Psi}}^{-1}(1\rtimes_{\widehat{\rho}}1^0 \rtimes_{\widehat{\widehat{\rho}}}h)) \\
& =\Psi(\phi(\widehat{\widehat{\Psi}}^{-1}(h)))=\phi(\widehat{\widehat{\Psi}}^{-1}(h))
\end{align*}
by Lemmas \ref{lem:equation2} and \ref{lem:equation3}.
Hence by the above equation, for any $h\in H$, $\phi\in H^0$,
\begin{align*}
S^0 (\lambda^0 (\phi))(h) & =\lambda^0 (\phi)(S(h))=\overline{\lambda^0 (\phi)^* (h^* )}
=\overline{\widehat{\Psi}(\phi^* )(h^* )}=\overline{\lambda^0 (\phi^* )(h^* )} \\
& =\overline{\phi^*(\widehat{\widehat{\Psi}}^{-1}(h^* ))}=\phi(S(\widehat{\widehat{\Psi}}^{-1}(h)))
=S^0 (\phi)(\widehat{\widehat{\Psi}}^{-1}(h)) =\lambda^0 (S^0 (\phi))(h) .
\end{align*}
Thus $S^0 (\lambda^0 (\phi))=\lambda^0 (S^0 (\phi))$ for any $\phi\in H^0 $.
Also, for any $h, l\in H$, $\phi\in H^0$,
\begin{align*}
\Delta^0 (\lambda^0 (\phi))(h\otimes l) & = \lambda^0 (\phi)(hl)=\phi(\widehat{\widehat{\Psi}}^{-1}(hl))
=\Delta^0 (\phi)(\widehat{\widehat{\Psi}}^{-1}(h)\otimes\widehat{\widehat{\Psi}}^{-1}(l)) \\
& =\phi_{(1)}(\widehat{\widehat{\Psi}}^{-1}(h))\phi_{(2)}(\widehat{\widehat{\Psi}}^{-1}(l))
=\lambda^0 (\phi_{(1)})(h)\lambda^0 (\phi_{(2)})(l) \\
& =(\lambda^0 (\phi_{(1)})\otimes\lambda^0 (\phi_{(2)}))(h\otimes l) \\
& =(\lambda^0 \otimes\lambda^0 )(\Delta^0 (\phi))(h\otimes l) .
\end{align*}
Hence $\Delta^0 (\lambda^0 (\phi))=(\lambda^0 \otimes\lambda^0 )(\Delta^0 (\phi))$ for any $\phi\in H^0$.
Furthermore, for any $\phi\in H^0 $,
$$
\epsilon^0 (\lambda^0 (\phi))=\lambda^0 (\phi)(1)=\phi(\widehat{\widehat{\Psi}}^{-1}(1))=\phi(1)=\epsilon(1) .
$$
Therefore, we obtain the conclusion.
\end{proof}

\begin{lemma}\label{lem:conjugate}With the above notations,
$$
\widehat{\widehat{\rho}}\circ\widehat{\Psi}=(\widehat{\Psi}\otimes\lambda^0 )\circ\widehat{\widehat{\sigma}}
$$
on $D_1$.
\end{lemma}
\begin{proof}For any $x\in D$,
\begin{align*}
(\widehat{\widehat{\rho}}\circ\widehat{\Psi})(x\rtimes_{\widehat{\sigma}}1^0 ) &=\widehat{\widehat{\rho}}
(\Psi(x)\rtimes_{\widehat{\rho}}1^0 )= (\Psi(x)\rtimes_{\widehat{\sigma}}1^ 0)\otimes 1^0
=\widehat{\Psi}(x\rtimes_{\widehat{\sigma}}1^0 )\otimes 1^0 \\
& =(\widehat{\Psi}\otimes\lambda^0 )((x\rtimes_{\widehat{\sigma}}1^0 )\otimes 1^0 )
=((\widehat{\Psi}\otimes\lambda^0 )\circ\widehat{\widehat{\sigma}})(x\rtimes_{\widehat{\sigma}}1^0 )
\end{align*}
by Lemma \ref{lem:equation2}. Also, for any $\phi\in H^0$,
\begin{align*}
(\widehat{\widehat{\rho}}\circ\widehat{\Psi})(1\rtimes_{\widehat{\sigma}}\phi) & =
\widehat{\widehat{\rho}}(1\rtimes_{\widehat{\rho}}\lambda^0 (\phi))=
(1\rtimes_{\widehat{\rho}}\lambda^0 (\phi)_{(1)})\otimes\lambda(\phi)_{(2)} \\
& =(1\rtimes_{\widehat{\rho}}\lambda^0 (\phi_{(1)}))\otimes\lambda^0 (\phi_{(2)}) 
=\widehat{\Psi}(1\rtimes_{\widehat{\sigma}}\phi_{(1)})\otimes\lambda^0 (\phi_{(2)}) \\
& =(\widehat{\Psi}\otimes\lambda^0 )((1\rtimes_{\widehat{\sigma}}\phi_{(1)})\otimes\phi_{(2)})
=((\widehat{\Psi}\otimes\lambda^0 )\circ\widehat{\widehat{\sigma}})(1\rtimes_{\widehat{\sigma}}\phi)
\end{align*}
by Lemma \ref{lem:Hopfiso}. Therefore, we obtain that
$\widehat{\widehat{\rho}}\circ\widehat{\Psi}=(\widehat{\Psi}\otimes\lambda^0 )\circ\widehat{\widehat{\sigma}}$.
\end{proof}

\begin{prop}\label{prop:key2}Let $(\rho, u)$ and $(\sigma, v)$ be twisted coactions of $H^0$ on $A$.
Let $C=A\rtimes_{\rho, u}H$ and $D=A\rtimes_{\sigma, v}H$.
We suppose that there is an isomorphism $\Psi$ of $C$ onto $D$ satisfying that
$$
\Psi|_A =\id_A , \quad E_1^{\rho, u}\circ\Psi=E_1^{\sigma, v} ,
$$
where $E_1^{\rho, u}$ and $E_1^{\sigma, v}$ are the canonical conditional expectations from $C$ and $D$
onto $A$ defined by
$$
E_1^{\rho, u}(a\rtimes_{\rho, u}h)=\tau(h)a , \quad E_1^{\sigma, v}(a\rtimes_{\sigma, v}h)=\tau(h)a
$$
for any $a\in A$, $h\in H$, respectively. Furthermore, we suppose that $A' \cap C=\BC1$. Then there is a
$C^*$-Hopf algebra automorphism $\lambda^0$ of $H^0$ such that $(\rho, u)$ is
strongly Morita equivalent to the twisted coaction
$((\id_A \otimes \lambda^0 )\circ \sigma,\, (\id_A \otimes \lambda^0 \otimes \lambda^0 )(v))$.
\end{prop}
\begin{proof}By \cite [Theorem 3.3]{KT2:coaction}, there are isomorphism
of $\Phi_{\sigma}$ of $A\otimes M_N (\BC)$ onto $D_1$ and a unitary element
$U_{\sigma}\in D_1 \otimes H^0$ such that
\begin{align*}
\Ad (U_{\sigma})\circ\widehat{\widehat{\sigma}} & =(\Phi_{\sigma}\otimes\id_{H^0})
\circ(\sigma\otimes\id_{M_N (\BC)})\circ\Phi_{\sigma}^{-1} , \\
(\Phi_{\sigma}\otimes\id_{H^0}\otimes\id_{H^0})(v\otimes I_N) & =(U_{\sigma}\otimes 1^0 )
(\widehat{\widehat{\sigma}}\otimes\id_{H^0})(U_{\sigma})(\id\otimes\Delta^0 )(U_{\sigma}^* ) ,
\end{align*}
where we identify $M_N (\BC)\otimes H^0$ with $H^0 \otimes M_N (\BC)$. Thus by Lemma \ref{lem:exterior},
$(\id_{D_1}\otimes\lambda^0 )\circ\widehat{\widehat{\sigma}}$ is exterior equivalent to
the twisted coaction
\begin{align*}
((\id_{D_1}\otimes\lambda^0 )& \circ(\Phi_{\sigma}\otimes\id_{H^0 })
\circ(\sigma\otimes\id_{M_N (\BC)})\circ\Phi_{\sigma}^{-1} \, ,  \\
& (\id_{D_1}\otimes\lambda^0 \otimes\lambda^0 )((\Phi_{\sigma}\otimes\id_{H^0}\otimes\id_{H^0})(v\otimes I_N ))) ,
\end{align*}
where $\lambda^0 $ is the $C^*$-Hopf automorphism of $H^0$ defined before Lemma \ref{lem:Ciso}.
Since
\begin{align*}
& (\id_{D_1}\otimes\lambda^0 ) \circ(\Phi_{\sigma}\otimes\id_{H^0 })
\circ(\sigma\otimes\id_{M_N (\BC)})\circ\Phi_{\sigma}^{-1} \\
& =(\Phi_{\sigma}\otimes\id_{H^0})\circ(\id_{A\otimes M_N (\BC)}\otimes\lambda^0 )
\circ(\sigma\otimes\id_{M_N (\BC)})\circ\Phi_{\sigma}^{-1},
\end{align*}
and
\begin{align*}
& (\id_{D_1}\otimes\lambda^0 \otimes\lambda^0 )
((\Phi_{\sigma}\otimes\id_{H^0}\otimes\id_{H^0})(v\otimes I_N ))) \\
& =(\Phi_{\sigma}\otimes\id_{H^0}\otimes\id_{H^0})
((\id_{A\otimes M_N (\BC)}\otimes\lambda^0 \otimes\lambda^0 )(v\otimes I_N )) ,
\end{align*}
by Lemma \ref{lem:iso}, the twisted coaction
$$
((\id_{A\otimes M_N (\BC)}\otimes\lambda^0 )
\circ(\sigma\otimes\id_{M_N (\BC)}) \, , \,
(\id_{A\otimes M_N (\BC)}\otimes\lambda^0 \otimes\lambda^0 )(v\otimes I_N ))
$$
is strongly Morita equivalent to the twisted coaction
\begin{align*}
((\id_{D_1}\otimes\lambda^0 )& \circ(\Phi_{\sigma}\otimes\id_{H^0 })
\circ(\sigma\otimes\id_{M_N (\BC)})\circ\Phi_{\sigma}^{-1} \, ,  \\
& (\id_{D_1}\otimes\lambda^0 \otimes\lambda^0 )((\Phi_{\sigma}\otimes\id_{H^0}\otimes\id_{H^0})(v\otimes I_N ))) ,
\end{align*}
Hence $(\id_{D_1}\otimes\lambda^0 )\circ\widehat{\widehat{\sigma}}$ is strongly Morita equivalent to
$$
((\id_{A\otimes M_N (\BC)}\otimes\lambda^0 )
\circ(\sigma\otimes\id_{M_N (\BC)}) \, , \,
(\id_{A\otimes M_N (\BC)}\otimes\lambda^0 \otimes\lambda^0 )(v\otimes I_N )) .
$$
Thus by \cite [Lemma 4.10]{Kodaka:equivariance}, $(\id_{D_1}\otimes\lambda^0 )\circ\widehat{\widehat{\sigma}}$
is strongly Morita equivalent to
$$
((\id_A \otimes\lambda^0 )\circ\sigma \, , \, 
(\id_A \otimes\lambda^0 \otimes \lambda^0 )(v)) .
$$
On the other hand, by Lemmas \ref{lem:exterior} and \ref{lem:conjugate},
$\widehat{\widehat{\rho}}$ is strongly Moriat equivalent to
$(\id_{D_1}\otimes\lambda^0 )\circ\widehat{\widehat{\sigma}}$.
Since $\widehat{\widehat{\rho}}$ is strongly Morita equivalent to
$(\rho, u)$ by Remark \ref{remark:review}, $(\rho, u)$ is strongly Morita equivalent to
$$
((\id_A \otimes\lambda^0 )\circ\sigma \, , \, 
(\id_A \otimes\lambda^0 \otimes \lambda^0 )(v)) .
$$
\end{proof}

\begin{cor}\label{cor:key3}Let $(\rho, u)$ and $(\sigma, v)$ be twisted coactions of $H^0$ on $A$.
Let $C=A\rtimes_{\rho, u}H$ and $D=A\rtimes_{\sigma, v}H$.
We regard $A$ as an $A-A$equivalent bimodule in the usual way.
We suppose that the unital inclusions $A\subset C$ and $A\subset D$ are
strongly Morita equivalent with respect to a $C-D$-equivalence bimodule
 $Y$ and its closed subspace $A$, that is, we assume that the $A-A$-equivalence
 bimodule $A$ is included in $Y$. Furthermore, we suppose that $A' \cap C=\BC1$. Then there is a
$C^*$-Hopf algebra automorphism $\lambda^0$ of $H^0$ such that $(\rho, u)$ is
strongly Morita equivalent to the twisted coaction
$$
((\id_A \otimes \lambda^0 )\circ\sigma ,\, (\id_A \otimes \lambda^0 \otimes \lambda^0 )(v)) .
$$
\end{cor}
\begin{proof}
This is immediate by Proposition \ref{prop:key2} and the discussions before Lemma \ref{lem:equation}.
\end{proof}

\section{The main result}\label{sec:main}In this section, we present the main result in the paper.
We recall the previous discussions: Let $A$ and $B$ be unital  $C^*$-algebras
and $H$ a finite dimensional $C^*$-Hopf algebra with its dual $C^*$-Hopf algebra $H^0$.
Let $(\rho, u)$ and $(\sigma, v)$ be twisted coactions of $H^0$ on $A$ and $B$, respectively.
Le $C=A\rtimes_{\rho, u}H$ and $D=B\rtimes_{\sigma, v}H$ and let $C_1 =C\rtimes_{\widehat{\rho}}H^0$
and $D_1 =D\rtimes_{\widehat{\sigma}}H^0$, where $\widehat{\rho}$ and $\widehat{\sigma}$
are the dual coactions of $H$ on $C$ and $D$ induced by $(\rho, u)$ and $(\sigma, v)$,
respectively. W suppose that the inclusions $A\subset C$ and $B\subset D$ are strongly
Morita equivalent with respect to a $C-D$-equivalence bimodule $Y$ and its closed subspace $X$
in the sense of \cite [Definition 2.1]{KT4:morita}. Furthermore, we suppose that
$A' \cap C=\BC1$. Then by Proposition \ref{prop:leftcoaction} and Lemma \ref{lem:action}
there are a coaction $\beta$ of $H$ on $D$ and a coaction $\mu$ of $H$ on $Y$
such that $(C, D, Y, \widehat{\rho}, \beta, \mu, H)$ is a covariant system.
Hence the coactions $\widehat{\rho}$ and $\beta$ of $H$ on $D$ are strongly
Morita equivalent. Moreover, by the discussions at the end of Section \ref{sec:coaction},
the unital inclusions $D\subset D_1 (=D\rtimes_{\widehat{\sigma}}H^0 )$ and
$D\subset D\rtimes_{\beta}H^0 $ are strongly Morita equivalent with respect to
a $D_1 -D\rtimes_{\beta}H^0$-equivalence bimodule
$W(=\widetilde{Y_1}\otimes_{C_1}(Y\rtimes_{\mu}H^0 ))$ and its closed subspace $D$,
where we regard $D$ as a $D-D$-equivalence bimodule in the usual way and we can also see
that $D$ is included in $W$
by the discussions at the end of Section \ref{sec:coaction}.
Since $A' \cap C=\BC1$, by \cite [Lemma 10.3]{KT4:morita} and \cite [the proof of Proposition 2.7.3]
{Watatani:index} $D' \cap D_1 =\BC1$.
Thus by Corollary \ref{cor:key3}, there is a $C^*$-Hopf algebra automorphism $\lambda$ of $H$
such that $\beta$ is strongly Morita equivalent to $(\id_D \otimes\lambda)\circ\widehat{\sigma}$.
Hence $\widehat{\rho}$ is strongly Morita equivalent to $(\id_D \otimes\lambda)\circ\widehat{\sigma}$.
Let $\lambda^0$ be the $C^*$-Hopf algebra automorphism of $H^0$ induced by $\lambda$, that is,
we define $\lambda^0$ as follows: For any $\psi\in H^0$, $h\in H$,
$$
\lambda^0 (\psi)(h)=\psi(\lambda^{-1}(h)) .
$$
Let
$$\sigma_{\lambda^0}=(\id_B \otimes\lambda^0 )\circ\sigma , \quad
v_{\lambda^0}=(\id_B \otimes\lambda^0 \otimes\lambda^0 )(v) .
$$
Then by Lemma \ref{lem:induced}, $(\sigma_{\lambda^0 }, v_{\lambda^0})$ is a twisted coaction
of $H^0$ on $B$. Let $D_{\lambda^0} =B\rtimes_{\sigma_{\lambda^0}, v_{\lambda^0}}H$.

\begin{lemma}\label{lem:iso3}With the above notations,
let $\pi$ be the map from $D$ to $D_{\lambda^0}$ defined by
$$
\pi(b\rtimes_{\sigma, v}h)=b\rtimes_{\sigma_{\lambda^0}, v_{\lambda^0}}\lambda(h)
$$
for any $b\in B$, $h\in H$. Then $\pi$ is an isomorphism of $D$ onto $D_{\lambda^0}$ such that
$$
\widehat{\sigma_{\lambda^0}}\circ\pi=(\pi\otimes\id_{H})\circ(\id_D \otimes\lambda)\circ\widehat{\sigma} ,
$$
where $\widehat{\sigma_{\lambda^0}}$ is the dual coaction of $(\sigma_{\lambda^0}, v_{\lambda^0})$,
a coaction of $H$ on $D_{\lambda^0}$.
\end{lemma}
\begin{proof}This is immediate by routine computations.
\end{proof}

\begin{thm}\label{thm:main}Let $A$ and $B$ be unital $C^*$-algebras and $H$ a finite
dimensional $C^*$-Hopf algebra with its dual $C^*$-Hopf algebra $H^0$.
Let $(\rho, u)$ and $(\sigma, v)$ be twisted coactions of $H^0$ on $A$ and $B$, respectively.
Let $C=A\rtimes_{\rho, u}H$ and $D=B\rtimes_{\sigma ,v}H$. We suppose that
the unital inclusions $A\subset C$ and $B\subset D$ are strongly Morita equivalent. Also,
we suppose that $A' \cap C=\BC1$. Then there is a $C^*$-Hopf algebra automorphism
$\lambda^0 $ of $H^0$ such that the twisted coaction $(\rho, u)$ is strongly Morita equivalent
to the twisted coaction
$$
((\id_B \otimes \lambda^0 )\circ\sigma , \, (\id_B \otimes\lambda^0 \otimes\lambda^0 )(v)) ,
$$
induced by $(\sigma, v)$ and $\lambda^0$.
\end{thm}
\begin{proof}By the discussions at the beginning of this section, the dual coaction $\widehat{\rho}$
of $(\rho, u)$ is strongly Morita equivalent to $(\id_D \otimes\lambda)\circ\widehat{\sigma}$. Also,
by Lemmas \ref{lem:iso} and \ref{lem:iso3}, $(\id_D \otimes\lambda)\circ\widehat{\sigma}$ is
strongly Morita equivalent to $((\id_B \otimes \lambda^0 )\circ\sigma)\,\widehat{}$, the dual coaction
of the twisted coaction
$$
((\id_B \otimes\lambda^0 )\circ\sigma , \, (\id\otimes\lambda^0 \otimes\lambda^0 )(v)) ,
$$
where $\lambda^0$ is the $C^*$-Hopf algebra automorphism of $H^0$ induced by $\lambda$.
Hence by \cite [Corollary 4.8]{KT3:equivalence}, $(\rho, u)$ is strongly Morita equivalent to
$$
((\id_B \otimes\lambda^0 )\circ\sigma , \, (\id\otimes\lambda^0 \otimes\lambda^0 )(v)) ,
$$
induced by $(\sigma, v)$ and $\lambda^0$.
\end{proof}

In the rest of the paper, we show the inverse implication of Theorem \ref{thm:main}.

\begin{lemma}\label{lem:inverse1}Let $A\subset C$ and $B\subset D$ be
unital inclusions of unital $C^*$-algebras. We suppose that there is an isomorphism $\pi$
of $D$ onto $C$ such that $\pi|_B$ is an isomorphism of $B$ onto $A$.
Then $A\subset C$ and $B\subset D$ are strongly Morita equivalent.
\end{lemma}
\begin{proof}Let$Y_{\pi}$ be the $D-C$-equivalence bimodule induced by $\pi$,
that is, $Y_{\pi}=D$ as vector spaces and the left $D$-action on $Y_{\pi}$
and the left $D$-valued inner product are defined in the evident way.
We define the right $C$-action and the right $C$-valued inner product as
follows: For any $c\in C$ and $y, z\in Y_{\pi}$,
$$
y\cdot c=y\pi^{-1}(c), \quad \la y, z \ra_C=\pi(y^* z) .
$$
Then $B$ is a closed subset of $Y_{\pi}$ and $Y_{\pi}$ and $B$ satisfy Conditions (1), (2)
in \cite [Definition 2.1]{KT4:morita} by easy computations.
Therefore, $A\subset C$ and $B\subset D$ are strongly Morita equivalent.
\end{proof}

Let $(\rho, u)$ be a twisted coaction of $H^0$ on $A$ and
$\lambda^0$ a $C^*$-Hopf algebra automorphism of $H^0$.
Let $(\rho_{\lambda^0}, \, u_{\lambda^0})$ be the twisted coaction of $H^0$
on $A$ induced by $(\rho, u)$ and $\lambda^0$, that is
$$
\rho_{\lambda^0}=(\id\otimes\lambda^0 )\circ\rho, \quad u_{\lambda^0}=(\id\otimes\lambda^0 \otimes\lambda^0 )(u) .
$$

\begin{lemma}\label{lem:inverse2}With the above notations, the unital inclusions
$A\subset A\rtimes_{\rho, u}H$ and $A\subset A\rtimes_{\rho_{\lambda^0}, u_{\lambda^0}}H$
are strongly Morita equivalent.
\end{lemma}
\begin{proof}Let $\pi$ be the map from $A\rtimes_{\rho, u}H$ to $A\rtimes_{\rho_{\lambda^0}, u_{\lambda^0}}H$ defined by
$$
\pi(a\rtimes_{\rho, u}h)=a\rtimes_{\rho_{\lambda^0}, u_{\lambda^0}}\lambda(h)
$$
for any $a\in A$, $h\in H$, where $\lambda$ is the $C^*$-Hopf algebra automorphism
of $H$ induced by $\lambda^0$, that is,
$$
\lambda^0 (\psi)(h)=\psi(\lambda^{-1}(h))
$$
for any $\psi\in H^0$ and $h\in H$. By easy computations, $\pi$ is an
isomorphism of $A\rtimes_{\rho, u}H$ onto $A\rtimes_{\rho_{\lambda^0}, u_{\lambda^0}}H$.
Also, for any $a\in A$,
$$
\pi(a\rtimes_{\rho, u}1) =a\rtimes_{\rho_{\lambda^0}, u_{\lambda^0 }}\lambda(1)
=a\rtimes_{\rho_{\lambda^0}, u_{\lambda^0}}\lambda(1)=a\rtimes_{\rho_{\lambda^0}, u_{\lambda^0}}1 .
$$
Hence $\pi_A =\id_A$. Thus, by Lemma \ref{lem:inverse1}, we obtain the conclusion.
\end{proof}

\begin{cor}\label{cor:iff}Let $A$ and $B$ be unital $C^*$-algebras and $H$ a finite
dimensional $C^*$-Hopf algebra wit hits dual $C^*$-Hopf algebra $H^0$.
Let $(\rho, u)$ and $(\sigma, u)$ be twisted coaction of $H^0$ on $A$ and $B$,
respectively. Let $C=A\rtimes_{\rho, u}H$ and $D=B\rtimes_{\sigma, v}H$.
We suppose that $A' \cap C=\BC1$. Then the following conditions are equivalent:
\newline
$(1)$ The unital inclusions $A\subset C$ and $B\subset D$ are strongly
Morita equivalent,
\newline
$(2)$ There is a $C^*$-Hopf algebra automorphism $\lambda^0$ of $H^0$ such that
the twisted coaction $(\rho, u)$ is strongly Morita equivalent to the twisted coaction
$$
((\id_B \otimes\lambda^0 )\circ\sigma \, , \, (\id_B \otimes\lambda^0 \otimes\lambda^0 )(v))
$$
induced by $(\sigma, v)$ and $\lambda^0$.
\end{cor}
\begin{proof}This is immediate by Theorem \ref{thm:main} and Lemmas \ref{lem:inverse1}, \ref{lem:inverse2}.
\end{proof}

\end{document}